\documentclass[11pt]{amsart}
\usepackage{amssymb}
\usepackage{amsmath}
\usepackage[active]{srcltx}
\usepackage{t1enc}
\usepackage[latin2]{inputenc}
\usepackage{verbatim}
\usepackage{amsmath,amsfonts,amssymb,amsthm}
\usepackage[mathcal]{eucal}
\usepackage{enumerate}
\usepackage[centertags]{amsmath}
\usepackage{graphics}

\newtheorem{theorem}{Theorem}

\newtheorem{lemma}{Lemma}

\newtheorem{corollary}{Corollary}

\begin{document}
\author{L-E. Persson, G. Tephnadze and G. Tutberidze }
\title[Strong convergence of Fejér means]{Some new result on strong convergence  of Fejér means with respect to Vilenkin systems}

\address {L-E. Persson, Department of Computer Science and Computational Engineering, UiT -The Arctic University of Norway, P.O. Box 385, N-8505, Narvik, Norway and Department of Mathematics and Computer Science, Karlstad University, Sweden.}
\email{lars.e.persson@uit.no \ \ \ larserik.persson@kau.se}
\address {G. Tephnadze, The University of Georgia, School of Science and Technology, 77a Merab Kostava St, Tbilisi, 0128, Georgia.}
\email{g.tephnadze@ug.edu.ge}
\address{G.Tutberidze, The University of Georgia, School of science and technology, 77a Merab Kostava St, Tbilisi 0128, Georgia and Department of Computer Science and Computational Engineering, UiT - The Arctic University of Norway, P.O. Box 385, N-8505, Narvik, Norway.}
\email{giorgi.tutberidze1991@gmail.com}
\thanks{The research of third author is supported by Shota Rustaveli National Science Foundation grant no. PHDF-18-476.}

\date{}
\maketitle

\begin{abstract}
In this paper we discuss and prove some new strong convergence theorems for partial sums and Fejér means with respect to the Vilenkin system.
\end{abstract}

\date{}



\section{Introduction}

Concerning definitions and notations used in this introductions we refer to Sections 2. 

It is well-known (for details see e.g. \cite{AVD, gol, sws}) that Vilenkin system forms not basis in the space $L_{1}\left(G_m\right).$ Moreover, there is a function in the martingale
Hardy space $H_{1}\left( G_{m}\right) ,$ such that the partial sums of $f$
are not bounded in $L_{1}\left( G_{m}\right) $-norm. However (for details see e.g. \cite{tep12.tep13}), for all $p>0$ and $f\in H_{p}$, there exists an absolute constant $c_{p} $ such that 
\begin{equation}  \label{1ccs}
\left\Vert S_{M_k}f\right\Vert _{p}\leq c_{p}\left\Vert f\right\Vert
_{H_{p}}.
\end{equation}%

In Gát \cite{gat1} (see also Simon \cite{Si3}) the following strong convergence result was obtained for
all  $f\in H_{1}\left( G_{m}\right) :$%
\begin{equation*}
\underset{n\rightarrow \infty }{\lim }\frac{1}{\log n}\overset{n}{\underset{%
		k=1}{\sum }}\frac{\left\Vert S_{k}f-f\right\Vert _{1}}{k}=0,
\end{equation*}%
It follow that 
\begin{equation*}
\frac{1}{\log n}\overset{n}{\underset{k=1}{\sum }}\frac{\left\Vert S_{k}f\right\Vert _{1}}{k} \leq \left\Vert f\right\Vert _{H_1}, \  \ n=2,3,... 
\end{equation*}%
In  \cite{tut1} was proved that for any $f\in H_{1}$ there exists an absolute constant $c,$ such that
\begin{equation*}
\sup_{n\in \mathbb{N}}\frac{1}{n\log n}\overset{n}{\underset{k=1}{\sum }}\left\Vert S_{k}f\right\Vert _{1}\leq \left\Vert f\right\Vert_{H_1}, \  \ n=2,3,... 
\end{equation*}
Moreover for every nondecreasing
function $\varphi :\mathbb{N}_{+}\rightarrow \lbrack 1,$ $\infty )$ satisfying the condition
\begin{equation*}
\overline{\lim_{n\rightarrow \infty }}\frac{\log n}{\varphi _{n}}=+\infty .
\end{equation*}
there exists a function $f\in H_{1},$ such that
\begin{equation*}
\sup_{n\in \mathbb{N}}\frac{1}{n\varphi_n}\overset{n}{\underset{k=1}{\sum }}%
\left\Vert S_{k}f\right\Vert _{1}=\infty .
\end{equation*}

For the Vilenkin system Simon \cite{si1}
proved that there is an absolute constant $c_{p},$ depending only on $p,$
such that
\begin{equation*}
\overset{\infty }{\underset{k=1}{\sum }}\frac{\left\Vert S_{k}f\right\Vert
	_{p}^{p}}{k^{2-p}}\leq c_{p}\left\Vert f\right\Vert _{H_{p}}^{p},
\end{equation*}%
for all $f\in H_{p}\left( G_{m}\right) $, where $0<p<1.$ In \cite%
{tep4} was proved that for any nondecreasing function $\Phi :\mathbb{N}%
\rightarrow \lbrack 1,$ $\infty )$, satisfying the condition $\underset{%
	n\rightarrow \infty }{\lim }\Phi \left( n\right) =+\infty ,$ there exists a
martingale $f\in H_{p}\left( G_{m}\right) ,$ such that

\begin{equation}
\text{ }\underset{k=1}{\overset{\infty }{\sum }}\frac{\left\Vert
	S_{k}f\right\Vert _{L_{p,\infty }}^{p}\Phi \left( k\right) }{k^{2-p}}=\infty
,\text{ for }0<p<1.  \label{2c}
\end{equation}

Strong convergence theorems of two-dimensional partial sums was investigate
by Weisz \cite{We}, Goginava \cite{gg}, Gogoladze \cite{Go}, Tephnadze \cite
{tep3,tep6}, (see also \cite{smt}).

Weisz \cite{We2} considered the norm convergence of Fejér means of
Walsh-Fourier series and proved the following:

\textbf{Theorem W1\ (Weisz).} Let $p>1/2$ and $f\in H_{p}.$ Then 
\begin{equation*}
\left\Vert \sigma _{k}f\right\Vert _{p}\leq c_{p}\left\Vert f\right\Vert
_{H_{p}}.
\end{equation*}

Moreover, Weisz \cite{We2} (see also \cite{PT}) also proved that for all $p>0$ and $f\in H_{p}$,
there exists an absolute constant $c_{p} $ such that 
\begin{equation}  \label{1ccsimga}
\left\Vert \sigma _{M_k}f\right\Vert _{p}\leq c_{p}\left\Vert f\right\Vert
_{H_{p}}.
\end{equation}%
Theorem W1 implies that%
\begin{equation*}
\frac{1}{n^{2p-1}}\overset{n}{\underset{k=1}{\sum }}\frac{\left\Vert \sigma
_{k}f\right\Vert _{p}^{p}}{k^{2-2p}}\leq c_{p}\left\Vert f\right\Vert
_{H_{p}}^{p},\text{ \ \ \ }\left( 1/2<p<\infty \right).
\end{equation*}

If Theorem W1 should hold for $0<p\leq 1/2,$ then we would have
\begin{equation}  \label{2cc}
\overset{\infty}{\underset{k=1}{\sum }}\frac{\left\Vert \sigma
_{k}f\right\Vert _{p}^{p}}{k^{2-2p}}\leq c_{p}\left\Vert f\right\Vert
_{H_{p}}^{p},\text{ \ \ \ }\left(0<p<1/2\right)
\end{equation}
\begin{equation}
\frac{1}{\log n}\overset{n}{\underset{k=1}{\sum }}\frac{\left\Vert \sigma
_{k}f\right\Vert _{1/2}^{1/2}}{k}\leq c\left\Vert f\right\Vert
_{H_{1/2}}^{1/2}\text{ \ \ \ }  \  \ n=2,3,...  \label{2cca}
\end{equation}
and 
\begin{equation}  \label{2ccaa}
\frac{1}{n}\overset{n}{\underset{k=1}{\sum }}{\left\Vert \sigma
_{k}f\right\Vert _{1/2}^{1/2}}\leq c\left\Vert f\right\Vert _{H_{1/2}}^{1/2}.
\end{equation}

However, in \cite{tep1} (see also \cite{BGG,BGG2} and \cite{tep7,tep8,tep9,tep10,tep100}) it was proved
that the assumption $p>1/2$ in Theorem W1 is essential. In particular, there
exists a martingale $f\in H_{1/2}$ such that 
\begin{equation*}
\sup_{n\in \mathbb{N}}\left\Vert \sigma _{n}f\right\Vert _{1/2}=+\infty .
\end{equation*}

In \cite{bt1} (see also \cite{tep5}) it was proved that (\ref{2cc}) and (\ref%
{2cca}) hold though Theorem W1 is not true for $0<p\leq 1/2.$

Moreover, in \cite{bt1} it was proved that if $0<p<1/2$ and $\Phi :\mathbb{N}%
_{+}\rightarrow\lbrack 1,\infty )$ be any nondecreasing function satisfying
condition 
\begin{equation*}
\overline{\underset{k\rightarrow \infty }{\lim }}\frac{k^{2-2p}}{\Phi_k }%
=\infty,
\end{equation*}
then there exists a martingale $f\in H_{p},$ such that 
\begin{equation*}
\underset{m=1}{\overset{\infty }{\sum }}\frac{\left\Vert \sigma
_{m}f\right\Vert _{weak-L_{p}}^{p}}{\Phi_m}=\infty.
\end{equation*}

On the other hand, for the Walsh system (\ref{2ccaa}) does not hold (see \cite{tep5}%
). In particular, it was proved that there exists a martingale $f\in
H_{1/2},$ such that
\begin{equation}  \label{2ccaacc}
\underset{n\in \mathbf{\mathbb{N}
}}{\sup }\frac{1}{n}\underset{m=1}{\overset{n}{\sum }}\left\Vert \sigma _{m}f\right\Vert
_{1/2}^{1/2}=\infty.
\end{equation}

In this paper we prove more general result for bounded Vilenkin system. In special case we also obtain  (\ref{2ccaacc}).

This paper is organized as follows: in order not to disturb our discussions
later on some definitions and notations are presented in Section 2. For the proofs of the main results we need some auxiliary Lemmas, some of them are new and of independent interest. These results are presented in Section 3. The main result with proof is given in Section 4. \qquad

\section{Definitions and Notations}

Let $\mathbb{N}_{+}$ denote the set of the positive integers, $\mathbb{N}:=%
\mathbb{N}_{+}\cup \{0\}.$

Let $m:=(m_{0,}m_{1},\dots)$ denote a sequence of the positive integers not
less than 2.

Denote by 
\begin{equation*}
Z_{m_{k}}:=\{0,1,\dots,m_{k}-1\}
\end{equation*}
the additive group of integers modulo $m_{k}.$

Define the group $G_{m}$ as the complete direct product of the group $%
Z_{m_{j}}$ with the product of the discrete topologies of $Z_{m_{j}}$ $^{,}$%
s.

The direct product $\mu $ of the measures 
\begin{equation*}
\mu _{k}\left( \{j\}\right):=1/m_{k}\text{ \qquad }(j\in Z_{m_{k}})
\end{equation*}
is the Haar measure on $G_{m_{\text{ }}}$with $\mu \left( G_{m}\right) =1.$

If $\sup_{n\in \mathbb{N}}m_{n}<\infty $, then we call $G_{m}$ a bounded
Vilenkin group. If the generating sequence $m$ is not bounded then $G_{m}$
is said to be an unbounded Vilenkin group. In this paper we discuss
bounded Vilenkin groups only.

The elements of $G_{m}$ are represented by sequences 
\begin{equation*}
x:=(x_{0},x_{1},\dots,x_{k},\dots)\qquad \left( \text{ }x_{k}\in
Z_{m_{k}}\right).
\end{equation*}

It is easy to give a base for the neighbourhood of $G_{m}$ namely
\begin{equation*}
I_{0}\left( x\right):=G_{m},
\end{equation*}%
and
\begin{equation*}
I_{n}(x):=\{y\in G_{m}\mid y_{0}=x_{0},\dots,y_{n-1}=x_{n-1}\}\text{ }(x\in
G_{m},\text{ }n\in \mathbb{N})
\end{equation*}%
Denote $I_{n}:=I_{n}\left( 0\right) $ for $n\in \mathbb{N}$ and $\overline{%
I_{n}}:=G_{m}$ $\backslash $ $I_{n}$ $.$

Let

\begin{equation*}
e_{n}:=\left( 0,\dots,0,x_{n}=1,0,\dots\right) \in G_{m}\qquad \left( n\in 
\mathbb{N}\right).
\end{equation*}

If we define the so-called generalized number system based on $m$ in the
following way: 
\begin{equation*}
M_{0}:=1,\text{ \qquad }M_{k+1}:=m_{k}M_{k\text{ }},\ \qquad (k\in \mathbb{N})
\end{equation*}%
then every $n\in \mathbb{N}$ can be uniquely expressed as $%
n=\sum_{k=0}^{\infty }n_{j}M_{j}$, where $n_{j}\in Z_{m_{j}}$ $~(j\in \mathbb{%
N})$ and only a finite number of $n_{j}`$s differ from zero. Let $\left\vert
n\right\vert :=\max $ $\{j\in \mathbb{N};$ $n_{j}\neq 0\}.$

For the natural number $n=\sum_{j=1}^{\infty }n_{j}M_{j},$ we define%
\begin{equation*}
\delta _{j}=signn_{j}=sign\left( \ominus n_{j}\right) ,\text{ \ \ \ \ }%
\delta _{j}^{\ast }=\left\vert \ominus n_{j}-1\right\vert \delta _{j},
\end{equation*}%
where $\ominus $ is the inverse operation for $a_{k}\oplus b_{k}=\left(
a_{k}+b_{k}\right) $mod$m_{k}.$

We define functions $v$ and $v^{\ast }$ by 
\begin{equation*}
v\left( n\right) =\sum_{j=0}^{\infty }\left\vert \delta _{j+1}-\delta
_{j}\right\vert +\delta _{0},\text{ \ }v^{\ast }\left( n\right)
=\sum_{j=0}^{\infty }\delta _{j}^{\ast },
\end{equation*}

The $n$-th Lebesgue constant is defined in the following way 
\begin{equation*}
L_{n}=\left\Vert D_{n}\right\Vert _{1}.
\end{equation*}

The norm (or quasi norm) of the space $L_{p}(G_{m})$ is defined by \qquad
\qquad \thinspace\ 
\begin{equation*}
\left\Vert f\right\Vert _{p}:=\left( \int_{G_{m}}\left\vert f(x)\right\vert
^{p}d\mu (x)\right) ^{1/p}\qquad \left( 0<p<\infty \right) .
\end{equation*}

The space $weak-L_{p}\left( G_{m}\right) $ consists of all measurable
functions $f$ for which

\begin{equation*}
\left\Vert f\right\Vert _{weak-L_{p}(G_{m})}:=\underset{\lambda >0}{\sup }%
\lambda ^{p}\mu \left\{ f>\lambda \right\} <+\infty .
\end{equation*}

Next, we introduce on $G_{m}$ an orthonormal system which is called the
Vilenkin system.

At first define the complex valued function $r_{k}\left( x\right)
:G_{m}\rightarrow \mathbb{C},$ the generalized Rademacher functions, as 
\begin{equation*}
r_{k}\left( x\right):=\exp \left( 2\pi\imath x_{k}/m_{k}\right) \text{
\qquad }\left( \imath^{2}=-1,\text{ }x\in G_{m},\text{ }k\in \mathbb{N}%
\right).
\end{equation*}

Now define the Vilenkin system $\psi:=(\psi _{n}:n\in \mathbb{N})$ on $G_{m} 
$ as: 
\begin{equation*}
\psi _{n}\left( x\right):=\prod_{k=0}^{\infty }r_{k}^{n_{k}}\left( x\right) 
\text{ \qquad }\left( n\in \mathbb{N}\right).
\end{equation*}

Specially, we call this system the Walsh-Paley one if $m\equiv 2.$

The Vilenkin system is orthonormal and complete in $L_{2}\left( G_{m}\right)
\,$ (for details see e.g. \cite{AVD, sws, Vi}).

If $\ f\in L_{1}\left( G_{m}\right) $ then we can define Fourier coefficients,
partial sums of the Fourier series, Fejér means, Dirichlet and Fejér kernels
with respect to the Vilenkin system in the usual manner: 
\begin{eqnarray*}
\widehat{f}(k) &:&=\int_{G_{m}}f\overline{\psi }_{k}d\mu \text{\thinspace
\qquad\ \ \ \ }\left( \text{ }k\in \mathbb{N}\text{ }\right); \\
S_{n}f &:&=\sum_{k=0}^{n-1}\widehat{f}\left( k\right) \psi _{k}\ \text{%
\qquad\ \ }\left( \text{ }n\in \mathbb{N}_{+},\text{ }S_{0}f:=0\right); \\
\sigma _{n}f &:&=\frac{1}{n}\sum_{k=0}^{n-1}S_{k}f\text{ \qquad\ \ \ \ \ }%
\left( \text{ }n\in \mathbb{N}_{+}\text{ }\right); \\
D_{n} &:&=\sum_{k=0}^{n-1}\psi _{k\text{ }}\text{ \qquad\ \ \qquad }\left( 
\text{ }n\in \mathbb{N}_{+}\text{ }\right); \\
K_{n} &:&=\frac{1}{n}\overset{n-1}{\underset{k=0}{\sum }}D_{k}\text{ \qquad\
\ \ \thinspace }\left( \text{ }n\in \mathbb{N}_{+}\text{ }\right).
\end{eqnarray*}

Recall that  (for details see e.g. \cite{AVD} and \cite{gol})
\begin{equation}  \label{3}
\quad \hspace*{0in}D_{M_{n}}\left( x\right) =\left\{ 
\begin{array}{l}
\text{ }M_{n}\text{ \ \ \ }x\in I_{n} \\ 
\text{ }0\text{ \qquad }x\notin I_{n}%
\end{array}
\right.
\end{equation}
and 
\begin{equation}  \label{9dn}
D_{s_{n}M_{n}}=D_{M_{n}}\sum_{k=0}^{s_{n}-1}\psi
_{kM_{n}}=D_{M_{n}}\sum_{k=0}^{s_{n}-1}r_{n}^{k} ,\text{ \qquad }  1\leq
s_{n}\leq m_{n}-1.
\end{equation}

The $\sigma$-algebra generated by the intervals $\left\{ I_{n}\left(
x\right):x\in G_{m}\right\} $ will be denoted by $\digamma _{n}$ $\left(
n\in \mathbb{N}\right).$ Denote by $f=\left( f_{n},\ \ n\in \mathbb{N}%
\right) $ a martingale with respect to $\digamma _{n}$ $\left( n\in \mathbb{N%
}\right)$ (for details see e.g. \cite{We1}). The maximal function of a
martingale $f$ is defined by \qquad 
\begin{equation*}
f^{*}=\sup_{n\in \mathbb{N}}\left| f_n\right|.
\end{equation*}

In the case $f\in L_{1}(G_{m}),$ the maximal functions are also be given by 
\begin{equation*}
f^{*}\left( x\right) =\sup_{n\in \mathbb{N}}\frac{1}{\left| I_{n}\left(
x\right) \right| }\left| \int_{I_{n}\left( x\right) }f\left( u\right) \mu
\left( u\right) \right|.
\end{equation*}

For $0<p<\infty $ the Hardy martingale spaces $H_{p}\left( G_{m}\right)$
consist of all martingales for which 
\begin{equation*}
\left\| f\right\| _{H_{p}}:=\left\| f^{*}\right\| _{p}<\infty.
\end{equation*}

If $f\in L_{1}(G_{m}),$ then it is easy to show that the sequence $\left(
S_{M_{n}}f :n\in \mathbb{N}\right) $ is a martingale. If $f=( f_n,\ \ n\in \mathbb{N}) $ is martingale, then the Vilenkin-Fourier coefficients must be
defined in a slightly different manner: $\qquad \qquad $ 
\begin{equation*}
\widehat{f}\left( i\right) :=\lim_{k\rightarrow \infty
}\int_{G_{m}}f_k \left( x\right) \overline{\psi }_{i}\left(
x\right) d\mu \left( x\right) .
\end{equation*}

The Vilenkin-Fourier coefficients of $f\in L_{1}\left( G_{m}\right) $ are
the same as those of the martingale $\left( S_{M_{n}}f:n\in \mathbb{N}%
\right) $ obtained from $f$.

A bounded measurable function $a$ is p-atom, if there exist an interval $I$,
such that%
\begin{equation*}
\int_{I}ad\mu =0,\text{ \ \ }\left\Vert a\right\Vert _{\infty }\leq \mu
\left( I\right) ^{-1/p},\text{ \ \ supp}\left( a\right) \subset I.\qquad
\end{equation*}
\qquad
\section{Auxiliary Lemmas}

\begin{lemma}
\cite{We1,We3} \label{lemma2.1} A martingale $f=\left( f_{n}, \ \ n\in 
\mathbb{N}\right) $ is in $H_{p}\left( 0<p\leq 1\right) $ if and only if
there exist a sequence $\left( a_{k},k\in \mathbb{N}\right) $ of p-atoms and
a sequence $\left( \mu _{k},k\in \mathbb{N}\right) $ of real numbers such
that, for every $n\in \mathbb{N},$%
\begin{equation}
\qquad \sum_{k=0}^{\infty }\mu _{k}S_{M_{n}}a_{k}=f_{n},\text{ \ \ a.e.,}
\label{condmart}
\end{equation}%
where 
\begin{equation*}
\qquad \sum_{k=0}^{\infty }\left\vert \mu _{k}\right\vert ^{p}<\infty .
\end{equation*}%
Moreover, 
\begin{equation*}
\left\Vert f\right\Vert _{H_{p}}\backsim \inf \left( \sum_{k=0}^{\infty
}\left\vert \mu _{k}\right\vert ^{p}\right) ^{1/p}
\end{equation*}
where the infimum is taken over all decomposition of $f$ of the form (\ref%
{condmart}).
\end{lemma}

By using atomic decomposition of $f\in H_p $ martingales, we can derive a
counterexample, which play a central role to prove sharpness of our main
results and it will be used several times in the paper:

\begin{lemma}
\label{lemma3} \cite{ptw} Let $n\in \mathbb{N}$ and $1\leq s_{n}\leq m_{n}-1$%
. Then\bigskip 
\begin{equation*}
s_{n}M_{n}K_{s_{n}M_{n}}=\sum_{l=0}^{s_{n}-1}\left(
\sum_{t=0}^{l-1}r_{n}^{t}\right) M_{n}D_{M_{n}}+\left(
\sum_{l=0}^{s_{n}-1}r_{n}^{l}\right) M_{n}K_{M_{n}}
\end{equation*}%
and%
\begin{equation*}
\left\vert s_{n}M_{n}K_{s_{n}M_{n}}\left( x\right) \right\vert \geq \frac{%
M_{n}^{2}}{2\pi },\text{ \ \ \ for \ \ \ \ }x\in I_{n+1}\left(
e_{n-1}+e_{n}\right) .
\end{equation*}

Moreover, if $x\in I_{t}/I_{t+1},$ \ $x-x_{t}e_{t}\notin I_{n}$ and $n>t,$
then%
\begin{equation}  \label{100kn}
K_{s_{n}M_{n}}(x)=0.
\end{equation}
\end{lemma}

\begin{lemma}
\label{lemma4} \cite{bt1}Let $n=\sum_{i=1}^{r}s_{n_{i}}M_{n_{i}}$, where $%
n_{n_1}>n_{n_2}>\dots >n_{n_r}\geq 0$ and $1\leq s_{n_i}<m_{n_{i}}$ \ for
all $1\leq i\leq r$ as well as $n^{(k)}=n-\sum_{i=1}^{k}s_{n_i}M_{n_{i}}$,
where $0<k\leq r$. Then 
\begin{eqnarray*}
nK_{n}&=&\sum_{k=1}^{r}\left( \prod_{j=1}^{k-1}r_{n_{j}}^{s_{n_j}}\right)
s_{n_k}M_{n_{k}}K_{s_{n_k}M_{n_{k}}} \\
&+&\sum_{k=1}^{r-1}\left( \prod_{j=1}^{k-1}r_{n_{j}}^{s_{n_j}}\right)
n^{(k)}D_{s_{n_k}M_{n_{k}}}.
\end{eqnarray*}
\end{lemma}

\begin{lemma}
\label{lemma5} Let 
\begin{equation*}
n=\sum_{i=1}^{s}\sum_{k=l_{i}}^{m_{i}}n_{k}M_{k},
\end{equation*}
\ where 
\begin{equation*}
0\leq l_{1}\leq m_{1}\leq l_{2}-2<l_{2}\leq m_{2}\leq ...\leq
l_{s}-2<l_{s}\leq m_{s}.
\end{equation*}

Then 
\begin{equation*}
n\left\vert K_{n}\left( x\right) \right\vert \geq cM_{l_{i}}^{2},\text{ \ \
for \ \ }x\in I_{l_{i}+1}\left( e_{l_{i}-1}+e_{l_{i}}\right) ,
\end{equation*}%
where $\lambda =\sup_{n\in \mathbb{N}}m_{n}$ and $c$ is an absolute constant.
\end{lemma}

\begin{proof}
Let $x\in I_{l_{i}+1}\left( e_{l_{i}-1}+e_{l_{i}}\right) .$ By combining (%
\ref{100kn}) in Lemma \ref{lemma3}, (\ref{3}) and (\ref{9dn}) we obtain that 
\begin{equation*}
D_{l_{i}}=0
\end{equation*}
and 
\begin{equation*}
D_{s_{n_k}M_{s_{n_k}}}=K_{s_{n_k}M_{{s_{n_k}}}}=0, \ \ s_{n_k}> l_{i}.
\end{equation*}

Since $s_{n_1}>s_{n_2}>\dots >s_{n_r}\geq 0$ we find that 
\begin{eqnarray*}
n^{(k)}&=& n-\sum_{i=1}^{k}s_{n_i}M_{n_{i}}=\sum_{i=k+1}^{s}s_{n_i}M_{n_{i}} \\
&\leq & \sum_{i=0}^{n_{k+1}}(m_{i}-1)M_{i}= m_{n_{k+1}}M_{n_{k+1}}-1\leq M_{n_{k}}.
\end{eqnarray*}

According to Lemma \ref{lemma4} we have that 
\begin{eqnarray*}
n\left\vert K_{n}\right\vert &\geq &\left\vert
s_{l_{i}}M_{l_{i}}K_{s_{l_{i}}M_{l_{i}}}\right\vert \\
&-&\sum_{r=1}^{i-1}\sum_{k=l_{r}}^{m_{r}}\left\vert
s_{k}M_{k}K_{s_{k}M_{k}}\right\vert  \notag \\
&-&\sum_{r=1}^{i-1}\sum_{k=l_{r}}^{m_{r}}\left\vert
M_{k}D_{s_{k}M_{k}}\right\vert \\
&=& I_1-I_2-I_3.
\end{eqnarray*}

Let $x\in I_{l_{i}+1}\left( e_{l_{i}-1}+e_{l_{i}}\right) $ and $1\leq
s_{l_{i}}\leq m_{l_{i}}-1$. By using Lemma \ref{lemma3} we get that 
\begin{eqnarray}  \label{10.0}
I_1=\left\vert s_{l_{i}}M_{l_{i}}K_{s_{l_{i}}M_{l_{i}}}\right\vert
\geq \frac{M_{l_{i}}^{2}}{2\pi }\geq \frac{2M_{l_{i}}^{2}}{9}.  \notag
\end{eqnarray}

It is easy to see that%
\begin{eqnarray*}
\sum_{s=0}^{k}n_{s}^{2}M_{s}^{2}&\leq& \sum_{s=0}^{k}\left( m_{s}-1\right) ^{2}M_{s}^{2} \\
&\leq
&\sum_{s=0}^{k}m_{s}^{2}M_{s}^{2}-2\sum_{s=0}^{k}m_{s}M_{s}^{2}+%
\sum_{s=0}^{k}M_{s}^{2} \\
&=&\sum_{s=0}^{k}M_{s+1}^{2}-2\sum_{s=0}^{k}M_{s+1}M_{s}+%
\sum_{s=0}^{k}M_{s}^{2} \\
&=&
M_{k+1}^{2}+2\sum_{s=0}^{k}M_{s}^{2}-2\sum_{s=0}^{k}M_{s+1}M_{s}-M_{0}^{2} \\
&\leq & M_{k+1}^{2}-1.
\end{eqnarray*}%
and%
\begin{eqnarray*}
\sum_{s=0}^{k}n_{s}M_{s}\leq \sum_{s=0}^{k}\left( m_{s}-1\right) M_{s}= m_{k}M_{k}-m_{0}M_{0} \leq M_{k+1}-2.
\end{eqnarray*}

Since $m_{i-1}\leq l_{i}-2$ if we use the estimates above, then we obtain that

\begin{eqnarray}  \label{10.1}
I_2 &\leq &\sum_{s=0}^{l_{i}-2}\left\vert n_{s}M_{s}K_{n_{s}M_{s}}\left(
x\right) \right\vert \leq\sum_{s=0}^{l_{i}-2}n_{s}M_{s}\frac{\left( n_{s}M_{s}+1\right) }{2} \\ \notag
&\leq &\frac{\left( m_{l_{i}-2}-1\right) M_{l_{i}-2}}{2}\sum_{s=0}^{l_{i}-2}%
\left( n_{s}M_{s}+1\right) 
\end{eqnarray}

\begin{eqnarray}
&\leq &\frac{\left( m_{l_{i}-2}-1\right) M_{l_{i}-2}}{2}M_{l_{i}-1} + \frac{\left( m_{l_{i}-2}-1\right) M_{l_{i}-2}}{2}l_{i}  \notag \\
&\leq &\frac{M_{l_{i}-1}^{2}}{2}-\frac{M_{l_{i}-2}M_{l_{i}-1}}{2}%
+M_{l_{i}-1}l_{i}.  \notag
\end{eqnarray}

For $I_3$ we have that%
\begin{eqnarray}  \label{10.3}
 I_3 &\leq &\sum_{k=0}^{l_{i}-2}\left\vert M_{k}D_{n_{k}M_{k}}\left( x\right)
\right\vert\leq \sum_{k=0}^{l_{i}-2}n_{k}M_{k}^{2}  \notag \\
&\leq& M_{l_{i}-2}\sum_{k=0}^{l_{i}-2}n_{k}M_{k}  \leq M_{l_{i}-1}M_{l_{i}-2}-2M_{l_{i}-2}.  \notag
\end{eqnarray}

By combining (\ref{10.0})-(\ref{10.3}) we have that%
\begin{eqnarray*}
 n\left\vert K_{n}\left( x\right) \right\vert &\geq & I_1-I_2-I_3 \\
&\geq &\frac{M_{l_{i}}^{2}}{2\pi }+\frac{3}{2}+2M_{l_{i}-2} \\
&-&\frac{M_{l_{i}-1}M_{l_{i}-2}}{2}-\frac{M_{l_{i}-1}^{2}}{2}%
-M_{l_{i}-1}l_{i} \\
&\geq &\frac{M_{l_{i}}^{2}}{2\pi }-\frac{M_{l_{i}}^{2}}{16}-\frac{%
M_{l_{i}}^{2}}{8}+\frac{7}{2}-M_{l_{i}-1}l_{i} \\
&\geq &\frac{2M_{l_{i}}^{2}}{9}-\frac{3M_{l_{i}}^{2}}{16}+\frac{7}{2}%
-M_{l_{i}-1}l_{i} \\
&\geq &\frac{M_{l_{i}}^{2}}{144}-M_{l_{i}-1}l_{i}.
\end{eqnarray*}

Suppose that $l_{i}\geq 4$. Then 
\begin{eqnarray*}
 n\left\vert K_{n}\left( x\right) \right\vert &\geq& \frac{M_{l_{i}}^{2}}{36}-\frac{M_{l_{i}}}{4} \geq \frac{M_{l_{i}}^{2}}{36}-\frac{M_{l_{i}}^{2}}{64}\geq \frac{5M_{l_{i}}^{2}}{36\cdot 16} \geq \frac{M_{l_{i}}^{2}}{144}.
\end{eqnarray*}
The proof is complete.
\end{proof}

\qquad

\section{The Main Result}
Our main result reads:
\begin{theorem}
	\label{theorem1}a) Let $f\in H_{1/2}.$ Then there exists an absolute constant $%
	c,$ such that
	\begin{equation*}
	\sup_{n\in \mathbb{N}}\frac{1}{n\log n}\overset{n}{\underset{k=1}{\sum }}%
	\left\Vert\sigma_{k}f\right\Vert _{H_{1/2}}^{1/2}\leq c \left\Vert f\right\Vert _{H_{1/2}}^{1/2},\ \  \  \ n=2,3,... 
	\end{equation*}
	
	b) Let $\varphi :\mathbb{N}_{+}\rightarrow \lbrack 1,$ $\infty )$ be a nondecreasing
	function satisfying the condition
	\begin{equation}
	\overline{\lim_{n\rightarrow \infty }}\frac{\log n}{\varphi _{n}}=+\infty .
	\label{cond1}
	\end{equation}
	
	Then there exists a function $f\in H_{1/2},$ such that
	\begin{equation*}
	\sup_{n\in \mathbb{N}_+}\frac{1}{n\varphi_n}\overset{n}{\underset{k=1}{\sum }}%
		\left\Vert\sigma_{k}f\right\Vert _{H_{1/2}}^{1/2}=\infty .
	\end{equation*}
\end{theorem}

\begin{corollary}
	\label{corollary1}There exists a martingale $f\in H_{1/2},$ such that 
	\begin{equation*}
	\sup_{n\in \mathbb{N_+}}\frac{1}{n}\overset{n}{\underset{k=1}{\sum }}\left\Vert \sigma _{k}f\right\Vert _{1/2}^{1/2}=\infty .
	\end{equation*}
\end{corollary}

\qquad
\begin{proof}[Proof of Theorem \protect\ref{theorem1}]
	In \cite{tepthesis} was proved that there exists an absolute constant $c$, such that
	\begin{equation*}
\left\Vert\sigma_{k}f\right\Vert _{H_{1/2}}^{1/2}\leq c \log k \left\Vert f\right\Vert _{H_{1/2}}^{1/2}, \ \ k=1,2,... 
	\end{equation*} 
Hence,
\begin{equation*}
\frac{1}{n\log n}\overset{n}{\underset{k=1}{\sum }}%
\left\Vert \sigma_{k}f\right\Vert _{H_{1/2}}^{1/2}\leq  \frac{c\left\Vert f\right\Vert_{H_{1/2}}^{1/2}} {n\log n}\overset{n}{\underset{k=1}{\sum }}{\log k}\leq c\left\Vert f\right\Vert_{H_{1/2}}^{1/2}, \ \ n=2,3, \dots
\end{equation*}

The proof of part a) is complete.

Under the condition \eqref{cond1} there exists an increasing sequence of the positive integers $\left\{\alpha_{k}:k\in \mathbb{N}\right\} $ such that
\begin{equation*}
\overline{\lim_{k\rightarrow \infty }}\frac{\log M_{{\alpha _{k}}}}{\varphi
	_{2M_{\alpha _{k}}}}=+\infty
\end{equation*}%
and
\begin{equation} \label{69}
\sum_{k=0}^{\infty }\frac{\varphi _{2M_{\alpha _{k}}}^{1/2}}{\log ^{1/2}M_{{%
			\alpha _{k}}}}<c<\infty .  
\end{equation}

Let $f=( f_n,\ \ n\in \mathbb{N}) $ be martingale, defined by 
\begin{equation*}
f_n :=\sum_{\left\{ k;\text{ }2\alpha_{k}<n\right\} }\lambda _{k}a_{k},
\end{equation*}
where
\begin{equation*}
a_{k}=M_{\alpha_{k}}r_{\alpha _{k}}D_{M_{_{\alpha _{k}}}}=M_{\alpha_{k}}(D_{2M_{_{\alpha
			_{k}}}}-D_{M_{_{\alpha _{k}}}})
\end{equation*}%
and
\begin{equation*}
\lambda _{k}=\frac{\varphi _{2M_{\alpha _{k}}}}{\log M_{{\alpha
			_{k}}}}.
\end{equation*}

Since$\ $%
\begin{equation}
S_{2^{A}}a_{k}=\left\{ 
\begin{array}{ll}
a_{k}, & \alpha _{k} <A, \\ 
0, &  \alpha _{k} \geq A,%
\end{array}%
\right.  \label{4aa}
\end{equation}%
\begin{equation*}
\text{supp}(a_{k})=I_{\alpha _{k} },\text{\ }%
\int_{I_{\alpha _{k} }}a_{k}d\mu =0,\text{\ }\left\Vert
a_{k}\right\Vert _{\infty }\leq M^2_{\alpha _{k} }=\mu (%
\text{supp }a_{k})^{-2},
\end{equation*}%
if we apply Lemma \ref{lemma2.1} and (\ref{69}) we conclude that $f\in H_{1/2}.$

Moreover,
\begin{equation} \label{6}
\widehat{f}(j)=\left\{
\begin{array}{l}
M_{\alpha_{k}}{\lambda _{k}},\,\,\text{\ \ \ \ \thinspace \thinspace }j\in \left\{
M_{\alpha _{k}},...,2M_{\alpha _{k}}-1\right\} ,\text{ }k\in \mathbb{N} \\
0\text{ },\text{ \thinspace \qquad \thinspace \thinspace \thinspace
	\thinspace \thinspace }j\notin \bigcup\limits_{k=1}^{\infty }\left\{
M_{\alpha _{k}},...,2M_{\alpha _{k}}-1\right\} .\text{ }%
\end{array}%
\right.   
\end{equation}

We have that 
\begin{eqnarray}  \label{7}
\sigma _{n}f &=&\frac{1}{n}\sum_{j=0}^{M_{\alpha _{k}}-1}S_{j}f+\frac{1}{n}%
\sum_{j=M_{\alpha _{k}}}^{n-1}S_{j}f \\
&=&I+II.  \notag
\end{eqnarray}

Let $M_{\alpha _{k}}\leq j<2M_{\alpha _{k}}.$ Since
\begin{equation*}
D_{j+M_{\alpha _{k}}}=D_{M_{\alpha _{k}}}+\psi _{_{M_{\alpha _{k}}}}D_{j},%
\text{ \qquad when \thinspace \thinspace }j\leq M_{\alpha _{k}},
\end{equation*}
if we apply (\ref{6}) we obtain that
\begin{eqnarray}  \label{8sas}
S_{j}f &=&S_{M_{\alpha _{k}}}f+\sum_{v=M_{\alpha _{k}}}^{j-1}\widehat{f}%
(v)\psi _{v}  \\
&=&S_{M_{\alpha _{k}}}f+M_{\alpha_{k}}{\lambda _{k}}\sum_{v=M_{\alpha _{k}}}^{j-1}\psi _{v}
\notag \\
&=&S_{M_{\alpha _{k}}}f+M_{\alpha_{k}}{\lambda _{k}}\left( D_{j}-D_{M_{\alpha _{k}}}\right)
\notag \\
&=&S_{M_{\alpha _{k}}}f+{\lambda _{k}}\psi _{M_{\alpha _{k}}}D_{j-M_{\alpha
		_{k}}}  \notag
\end{eqnarray}

According to (\ref{8sas}) concerning $II$ we conclude can that 
\begin{eqnarray*}
II &=&\frac{n-M_{\alpha _{k}}}{n}S_{M_{\alpha_k}}f \\
&&+\frac{\lambda _{k}M_{\alpha _{k}}}{ n}\sum_{j=M_{2\alpha
_{k}}}^{n-1}\psi_{M_{\alpha_{k}}}D_{j-M_{\alpha_{k}}} \\
&&=II_{1}+II_{2}.
\end{eqnarray*}

We can estimate $II_{2}$ as fallows:
\begin{eqnarray*}
\left\vert II_{2}\right\vert &=&\frac{\lambda _{k}M_{\alpha _{k}}}{n}\left\vert \psi _{M_{\alpha _{k}}}\sum_{j=0}^{n-M_{\alpha
_{k}}-1}D_{j}\right\vert \\
&=&\frac{\lambda _{k}M_{\alpha _{k}}}{n}{(n-M_{\alpha _{k}})}%
\left\vert K_{n-M_{\alpha _{k}}}\right\vert \\
&\geq &{\lambda_{k}}\left( n-M_{\alpha _{k}}\right)
\left\vert K_{n-M_{\alpha _{k}}}\right\vert.
\end{eqnarray*}

Let $n=\sum_{i=1}^{s}\sum_{k=l_{i}}^{m_{i}}M_{k},$ \ where 
\begin{equation*}
0\leq l_{1}\leq m_{1}\leq l_{2}-2<l_{2}\leq m_{2}\leq ...\leq
l_{s}-2<l_{s}\leq m_{s}.
\end{equation*}

By applying Lemma \ref{lemma5} we get that 
\begin{eqnarray*}
\left\vert II_{2}\right\vert &\geq &{c\lambda _{k}\left\vert \left( n-M_{\alpha
_{k}}\right) K_{n-M_{\alpha _{k}}}\left( x\right) \right\vert } \\
&\geq &{c\lambda _{k}M_{l_{i}}^{2}},\text{ \ \ for \ \ }x\in
I_{l_{i}+1}\left( e_{l_{i}-1}+e_{l_{i}}\right) .
\end{eqnarray*}

Hence 
\begin{eqnarray}\label{8aaa0} 
\int_{G_{m}}\left\vert II_{2}\right\vert ^{1/2}d\mu  
&\geq &\sum_{i=1}^{s-1}\int_{I_{l_{i}+1}\left( e_{l_{i}-1}+e_{l_{i}}\right)
}\left\vert II_{2}\right\vert ^{1/2}d\mu   \\ \notag
&\geq &c\sum_{i=1}^{s-1}\int_{I_{l_{i}+1}\left( e_{l_{i}-1}+e_{l_{i}}\right)
}{\lambda _{k}^{1/2}M_{l_{i}}}d\mu   \\ \notag
&\geq &{c\lambda _{k}^{1/2}\left( s-1\right) } \geq{c\lambda _{k}^{1/2}v\left( n-M_{\alpha _{k}}\right) }.  \notag
\end{eqnarray}

In view of  (\ref{1ccs}), (\ref{1ccsimga}) and (\ref{7})  we find that 
\begin{eqnarray}  \label{8aaa}
\left\Vert I\right\Vert ^{1/2}&=&\left\Vert \frac{M_{\alpha _{k}}}{n}%
\sigma_{M_{\alpha _{k}}}f\right\Vert _{1/2}^{1/2} \leq  \left\Vert \sigma_{M_{\alpha _{k}}}f\right\Vert _{1/2}^{1/2}  \leq c\left\Vert f\right\Vert_{H_{1/2}}^{1/2} 
\end{eqnarray}
and 
\begin{eqnarray}  \label{8bbb}
\left\Vert II_{1}\right\Vert ^{1/2} 
&=&\left\Vert \frac{n-M_{\alpha _{k}}}{n}S_{M_{\alpha _{k}}}f\right\Vert
_{1/2}^{1/2}  \leq  \left\Vert S_{M_{\alpha _{k}}}f\right\Vert _{1/2}^{1/2}   \leq  c\left\Vert f\right\Vert_{H_{1/2}}^{1/2}.  
\end{eqnarray}

By combining (\ref{8aaa0}), (\ref{8aaa}) and (\ref{8bbb}) we get that 
\begin{eqnarray*}
\left\Vert \sigma _{n}f\right\Vert _{1/2}^{1/2} 
&\geq &\left\Vert II_{2}\right\Vert _{1/2}^{1/2}-\left\Vert
II_{1}\right\Vert _{1/2}^{1/2}-\left\Vert I\right\Vert _{1/2}^{1/2} \\
&\geq &{c\lambda _{k}^{1/2}v\left( {n-M_{\alpha _{k}}}\right) }%
-c\left\Vert f\right\Vert _{H_{1/2}}^{1/2}.
\end{eqnarray*}

By using estimates with the above we can conclude that 
\begin{eqnarray}
&&\underset{n\in \mathbf{\mathbb{N}}_{+}}{\sup }\frac{1}{n\varphi_n}\underset{k=1}{%
\overset{n}{\sum }}\left\Vert \sigma _{k}f\right\Vert _{1/2}^{1/2}
\label{8bbb100} \\
&\geq &\frac{1}{M_{\alpha _{k}+1}\varphi_{2M_{\alpha _{k}}}}\underset{\left\{ M_{\alpha _{k}}\leq
l\leq 2M_{\alpha _{k}}\right\} }{\sum }\left\Vert \sigma _{l}f\right\Vert
_{1/2}^{1/2}  \notag \\
&\geq &\frac{c}{M_{\alpha _{k}+1}\varphi_{2M_{\alpha _{k}}}}\underset{\left\{ M_{\alpha _{k}}\leq
l\leq 2M_{\alpha _{k}}\right\} }{\sum }\left( {\lambda _{k}^{1/2}v\left( l-M_{\alpha
_{k}}\right) }-c\left\Vert f\right\Vert
_{H_{1/2}}^{1/2}\right)  \notag \\
&\geq &\frac{c\lambda _{k}^{1/2}}{M_{\alpha _{k}}\varphi_{2M_{\alpha _{k}}}}\underset{l=1}{\overset{%
M_{\alpha _{k}}}{\sum }}v\left( l\right)-\frac{c\left\Vert f\right\Vert _{H_{1/2}}^{1/2}}{M_{\alpha _{k}}\varphi_{2M_{\alpha _{k}}}}%
\underset{\left\{ M_{\alpha _{k}}\leq l\leq 2M_{\alpha _{k}}\right\} }{\sum }%
1  \notag \\
&\geq &\frac{c\lambda _{k}^{1/2}}{M_{\alpha _{k}}\varphi_{2M_{\alpha _{k}}}}\underset{l=1}{\overset{%
M_{\alpha _{k}}-1}{\sum }}v\left( l\right)- c\geq c\frac{\log^{1/2} M_{{\alpha _{k}}}}{\varphi_{2M_{\alpha _{k}}}^{1/2}}\rightarrow \infty ,\text{ as \ }k\rightarrow
\infty .  \notag
\end{eqnarray}

The proof is complete.
\end{proof}


\begin{thebibliography}{99}

\bibitem{AVD} G. N. AGAEV, N. Ya. VILENKIN, G. M. DZHAFARLY and A. I.
RUBINSHTEIN, Multiplicative systems of functions and harmonic analysis on
zero-dimensional groups, Baku, Ehim, 1981 (in Russian).

\bibitem{blahota} I. BLAHOTA, Relation between Dirichlet kernels with
respect to Vilenkin-like systems, Acta Acad. Paed. Agriensis, XXII, 1994,
109-114.

\bibitem{BGG} I. BLAHOTA, G. GÁT and U. GOGINAVA, Maximal operators of Fejér
means of double Vilenkin-Fourier series, Colloq. Math., 107 (2007), no. 2,
287-296.

\bibitem{BGG2} I. BLAHOTA, G. GÁT and U. GOGINAVA, Maximal operators of Fejé%
r means of Vilenkin-Fourier series, J. Inequal. Pure Appl. Math., 7 (2006),
1-7.

\bibitem{bt1} I. BLAHOTA and G. TEPHNADZE, Strong convergence theorem for
Vilenkin-Fejér means, Publ. Math. Debrecen, 85 (1-2) (2014), 181--196.

\bibitem{gat1} G. GÁT, Inverstigations of certain operators with
respect to the Vilenkin sistem, Acta Math. Hung., 61 (1993), 131-149.

\bibitem{gat} G. GÁT, Ces\`{a}ro means of integrable functions with respect
to unbounded Vilenkin systems. J. Approx. Theory, 124 (2003), no. 1, 25-43.

\bibitem{GNCz} U. GOGINAVA and K. NAGY, On the maximal operator of
Walsh-Kaczmarz-Fejér means, Czechoslovak Math. J., 61 (2011), 3, 673-686.

\bibitem{gg}  U. GOGINAVA, L. D. GOGOLADZE, Strong Convergence of
Cubic Partial Sums of Two-Dimensional Walsh-Fourier series, Constructive
Theory of Functions, Sozopol 2010: In memory of Borislav Bojanov. Prof.
Marin Drinov Academic Publishing House, Sofia, 2012, pp. 108-117.

\bibitem{Go} L. D. GOGOLADZE, On the strong summability of Fourier
series, Bull of Acad. Scie. Georgian SSR, 52, 2 (1968), 287-292.


\bibitem{gol} B. I. GOLUBOV, A. V. EFIMOV and V. A. SKVORTSOV, Walsh series
and transforms, (Russian) Nauka, Moscow, 1987, English transl: Mathematics
and its Applications, 64. Kluwer Academic Publishers Group, Dordrecht, 1991.

\bibitem{luko} S. F. LUKOMSKII, Lebesgue constants for characters of the
compact zero-dimensional abelian group, East J. Approx. 15 (2009), no. 2,
219-231.

\bibitem{smt} N. MEMIĆ, I. SIMON and G. TEPHNADZE, Strong convergence of
two-dimensional Vilenkin-Fourier series, Math. Nachr., 289, 4 (2016) 485-500.

\bibitem{PS} J. PÁL and P. SIMON, On a generalization of the concept of
derivate, Acta Math. Hung., 29 (1977), 155-164.

\bibitem{PT}  L. E. PERSSON, G. TEPHNADZE, A sharp boundedness result concerning some maximal operators of Vilenkin-Fejér means, Mediterr. J. Math. 13,4(2016) 1841-1853.

\bibitem{sws} F. SCHIPP, W. R. WADE, P. SIMON and J. PÁL, Walsh series. An
introduction to dyadic harmonic analysis, Adam Hilger, Ltd., Bristol, 1990.

\bibitem{Si3} P. SIMON, Strong convergence of certain means with
respect to the Walsh-Fourier series, Acta Math. Hung., 49 (1-2) (1987),
425-431.

\bibitem{si1} P. SIMON, Strong convergence Theorem for
Vilenkin-Fourier Series. Journal of Mathematical Analysis and Applications,
245, (2000), pp. 52-68 .

\bibitem{ptw} L-E. PERSSON, G.TEPHNADZE and P. WALL, Some new $\left(
H_{p},L_{p}\right) $ type inequalities of maximal operators of Vilenkin-Nö%
rlund means with non-decreasing coefficients, J. Math. Inequal, 9, 4 (2015), 1055-1069.

\bibitem{Sc} F. SCHIPP, Certain rearrangements of series in the Walsh
series, Mat. Zametki, 18 (1975), 193-201.


\bibitem{tep1} G.TEPHNADZE, Fejér means of Vilenkin-Fourier series, Stud.
sci. Math. Hung., 49 (1), (2012) 79-90.

\bibitem{tep2} G.TEPHNADZE, On the maximal operator of Vilenkin-Fejér means,
Turk. J. Math, 37, (2013), 308-318.

\bibitem{tep3} G.TEPHNADZE, A note on the strong convergence of two-dimensional Walsh-Fourier series,  Transactions of A. Razmadze Math. Inst., 162 (2013), 93-97.

\bibitem{tepthesis} G.TEPHNADZE, Martingale Hardy Spaces and Summability of
the One Dimensional Vilenkin-Fourier Series, PhD thesis, Department of
Engineering Sciences and Mathematics, Lule\aa\ University of Technology,
Oct. 2015 (ISSN 1402-1544).

\bibitem{tep4} G. TEPHNADZE, A note on the Fourier coefficients and
partial sums of Vilenkin-Fourier series, Acta Math. Acad. Paed. Nyíreg., 28 (2012),167-176.

\bibitem{tep5} G.TEPHNADZE, Strong convergence theorems of Walsh-Fejér
means, Acta Math. Hungar., 142 (1) (2014), 244--259.

\bibitem{tep6} G.TEPHNADZE, Strong convergence of two-dimensional Walsh-Fourier series, Ukr. Math. J., 65, 6 (2013), 822-834.

\bibitem{tep7} G.TEPHNADZE, On The maximal operators of Vilenkin-Fejér means on Hardy spaces, Math. Inequal. Appl., 16, 2 (2013), 301-312.

\bibitem{tep8} G.TEPHNADZE, On the maximal operators of Vilenkin-Fejér means, Turk. J. Math., 37, (2013), 308-318.

\bibitem{tep9} G.TEPHNADZE,  A note on the norm convergence by Vilenkin-Fejér means, Georgian Math. J., 21, 4 (2014), 511-517. 

\bibitem{tep10} G.TEPHNADZE, Approximation by Walsh-Kaczmarz-Fejér means on the Hardy space, Acta Math. Sci., 34, 5 (2014), 1593-1602.

\bibitem{tep100} G.TEPHNADZE, On the maximal operators of Walsh-Kaczmarz-Fejér means, Period. Math. Hung., 67, 1 (2013), 33-45.

\bibitem{tep11} G.TEPHNADZE, On the convergence of Fejér means of Walsh-Fourier series in the space $H_{p},$ J. Contemp. Math. Anal., 51, 2 (2016), 90-102.

\bibitem{tep12} G.TEPHNADZE, On the partial sums of Walsh-Fourier series, Colloq. Math., 141, 2 (2015), 227-242.

\bibitem{tep13} G.TEPHNADZE, On the convergence of partial sums with respect to Vilenkin system on the martingale Hardy spaces, J. Contemp. Math. Anal., 53, 5, (2018) 294-306.

\bibitem{tut1} G. TUTBERIDZE, A note on the strong convergence of partial
sums with respect to Vilenkin system J. Contemp. Math. Anal., 54, 6, (2019), 319-324.

\bibitem{Vi} N. Ya. VILENKIN, On a class of complete orthonormal systems,
Izv. Akad. Nauk. U.S.S.R., Ser. Mat., 11 (1947), 363-400.

\bibitem{We1} F. WEISZ, Martingale Hardy spaces and their applications in
Fourier Analysis, Springer, Berlin-Heideiberg-New York, 1994.

\bibitem{We3} F. WEISZ, Hardy spaces and Ces\`{a}ro means of two-dimensional
Fourier series, Bolyai Soc. Math. Studies, (1996), 353-367.

\bibitem{We} F. WEISZ, Strong convergence theorems for two-parameter
Walsh-Fourier and trigonometric-Fourier series. (English) Stud. Math. 117,
No.2, (1996), 173-194.

\bibitem{We2} F. WEISZ, Ces\`{a}ro summability of one and two-dimensional
Fourier series, Anal. Math., 5 (1996), 353-367.
\end{thebibliography}
\end{document}